\documentclass[a4paper,12pt]{article}
\usepackage{amsmath, amsthm, amsfonts}
\usepackage{graphicx} 
\usepackage{color}
\usepackage{blindtext}
\usepackage{hyperref}
\usepackage{cancel}
\usepackage{geometry}
\usepackage{changepage}
\usepackage{caption}
\usepackage{xcolor}
\usepackage{enumitem}
\usepackage{soul}

\newtheorem{theorem}{Theorem}[section]

\newtheorem{lemma}[theorem]{Lemma}

\newtheorem{conj}[theorem]{Conjecture}

\theoremstyle{definition}

\DeclareMathOperator{\dist}{dist}

\DeclareMathOperator{\length}{length}

\DeclareMathOperator{\id}{Id}
\DeclareMathOperator{\diam}{diam}
\DeclareMathOperator{\Vol}{Vol}
\DeclareMathOperator{\Ric}{Ric}

\newcommand{\TODO}[1]{}

\renewcommand{\tilde}[1]{{\widetilde{#1}}}
\renewcommand{\th}{\theta}
\newcommand{\p}{\partial}

\renewcommand{\epsilon}{\varepsilon}

\title{Scalar curvature and volume entropy of hyperbolic 3-manifolds}
\author{Demetre Kazaras, Antoine Song, Kai Xu}
\date{}

\begin{document}

\maketitle

\begin{abstract} 
We show that any closed hyperbolic $3$-manifold $M$ admits a Riemannian metric with scalar curvature at least $-6$, but with volume entropy strictly larger than $2$. In particular, this construction gives counterexamples to a conjecture of I. Agol, P. Storm and W. Thurston.
\end{abstract}

A classical topic in Riemannian geometry is to understand how curvature and volume 
interact. For instance, an elementary application of the Gauss-Bonnet formula implies that if $(S,g_0)$ is a closed hyperbolic surface, then any Riemannian metric $g$ on $S$ with scalar curvature bound $R_g\geq -2$ has area at least that of $g_0$. A deep corollary of G. Perelman's resolution of the Geometrization conjecture of W. Thurston \cite{Anderson06,KL08} is that the corresponding bound still holds in dimension $3$: if $(M,g_0)$ is a closed hyperbolic $3$-manifold, for any metric $g$ on $M$, 
\begin{equation}\label{perelman}
\text{if}\quad R_g\geq -6,\quad \text{then}\quad\Vol(M,g) \geq \Vol(M,g_0).
\end{equation}
An earlier conjecture of R. Schoen \cite{Schoen06} predicts that this should hold in any dimension. For related 
local results, see \cite{BCG91,Yuan23}.

\TODO{added heuristic definition of scalar curvature, for comparison}
While scalar curvature provides information on the volumes of small geodesic balls,
the volume entropy of a closed Riemannian manifold $(M,g)$ provides a measure of the volumes of large balls in its universal cover,
given by the limit \TODO{added definition}
\begin{equation}
    h(g)=\lim_{r\to\infty}\frac{\log\left(\Vol(B_{\tilde{g}}(p_0,r),\tilde{g})\right)}{r}
\end{equation}
where $\Vol(B_{\tilde{g}}(p_0,r),\tilde{g})$ is the volume of a ball of radius $r$ in the universal cover $(\tilde{M},\tilde{g})$.
Bishop-Gromov's inequality readily shows that a manifold with Ricci curvature bounded from below $\Ric_g\geq -(n-1)g$ satisfies the volume entropy bound $h(g)\leq n-1$. On the other hand, a well-known theorem of G. Besson, G. Courtois and S. Gallot \cite{BCG95} states that if $(M,g_0)$ is a hyperbolic manifold of dimension $n\geq 3$ and $g$ is any metric on $M$, we have
\begin{equation}\label{BCG}
\text{if} \quad h(g)\leq n-1,\quad \text{then}\quad  \Vol(M,g) \geq \Vol(M,g_0). 
\end{equation}
Comparing statements \eqref{perelman} and \eqref{BCG}, one is led to examine the relation between scalar curvature and volume entropy.
In \cite[Conjecture 12.2]{AST07}, I. Agol, P. Storm and W. Thurston conjectured the following:
\begin{conj} \label{AST}
If $(M,g)$ is a closed Riemannian 3-manifold with $R_g \geq -6$, then $h(g)\leq 2$.
\end{conj}


In \cite{AST07}, the authors established certain volume bounds for $3$-manifolds with minimal boundary and scalar curvature lower bounds. Their method of proof heavily relied on Perelman's proof of \eqref{perelman}, and Conjecture \ref{AST} was actually partially motivated by the desire to find a proof which did not rely on Ricci flow with surgery. In particular, Besson-Courtois-Gallot's volume entropy inequality \eqref{BCG} and Conjecture \ref{AST} would give a different proof of \eqref{perelman}.

\vspace{12pt}

The main goal of this paper is to construct counterexamples to Conjecture \ref{AST}, and investigate how flexible the volume entropy is with respect to scalar curvature.

\begin{theorem}
\label{main}
    Let $(M,g_0)$ be a closed hyperbolic manifold of dimension $3$. Then $M$ admits a Riemannian metric $g$ with $R_g\geq -6$, but with volume entropy $h(g)$ strictly larger than $2$. In fact, $g$ can be chosen so that one of the following two conditions holds:
    
    (A) $h(g)$ is arbitrarily large,
    
    (B) $g$ is arbitrarily close to $g_0$ in the $C^0$ topology.
\end{theorem}

{ This theorem provides two different types of counterexamples, which actually correspond to the two extreme cases of a family of counterexamples. We note that (A) and (B) in the theorem cannot be satisfied at the same time, since $C^0$ closeness of the metrics implies closeness of the respective volume entropies by an elementary comparison.}




\vspace{12pt}


\textbf{Overview of proof:} 
The proof of Theorem \ref{main} can be summarized as follows. We first deform the given hyperbolic $3$-manifold $(M,g_0)$ using a drawstring 
construction adapted from \cite{KX}. Starting with a closed simple non-contractible geodesic $\gamma$, there exists a metric $g$ such that $R_g\geq-6 -o(1)$, and $g=g_0$ outside an arbitrarily thin neighborhood of $\gamma$, and the $g$-length of $\gamma$ is equal to a chosen value in the range $(0,\length_{g_0}(\gamma))$. This construction is the content of Theorem \ref{thm:creating_ds} in Section \ref{drawst}. Heuristically, the drawstring has the effect that one can reach more fundamental domains in the universal cover within the same distance, thereby increasing the volume of geodesic balls as depicted in Figure \ref{fig:pulledstringball}. To achieve condition (A) of the theorem, we choose $\length_{g}(\gamma)$ to be small. Then we show that $h(g)$ is large by applying a systolic bound for the volume entropy which was established under various forms by G. Besson, G. Courtois, S. Gallot, A. Sambusetti \cite{BCGS17}, F. Cerocchi \cite{Cerocchi14}, F. Balacheff, L. Merlin \cite{BM23}. To construct examples satisfying (B) of the theorem, we choose $\length_{g}(\gamma)$ to be close to $\length_{g_0}(\gamma)$, then from the construction, $g$ is $C^0$-close to $g_0$. Then we show that $h(g)$ is strictly larger than $2$ by applying a recent volume entropy comparison theorem \cite{S}.

\begin{figure}[h]
\begin{picture}(0,0)
\end{picture}
\begin{center}
\includegraphics[totalheight=5cm]{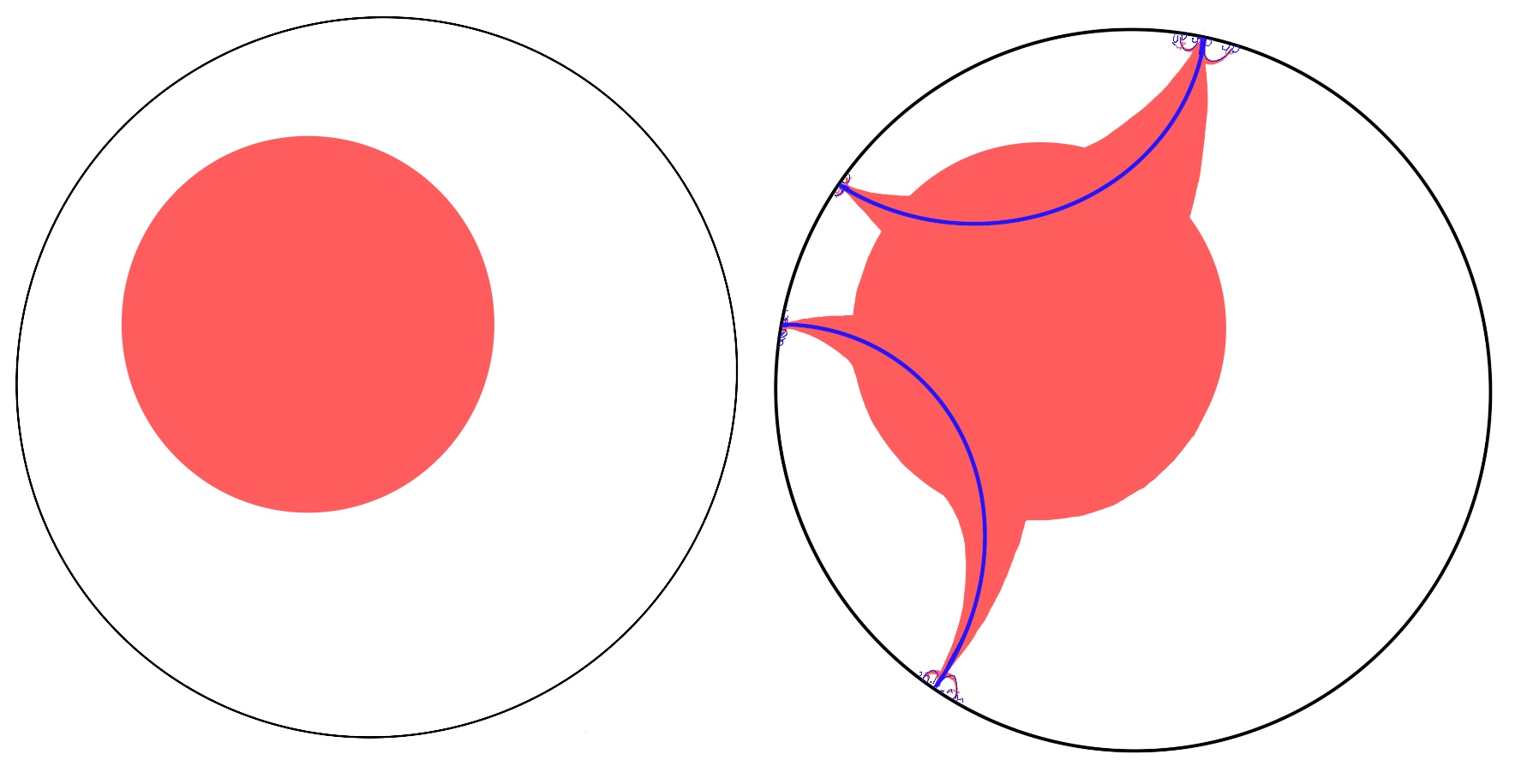}
\end{center}
\setlength{\abovecaptionskip}{-10pt}
\captionsetup{width=.85\linewidth}
\caption{On the left, a geodesic ball is depicted in the Poincar{\'e} disc model of hyperbolic space. On the right, the associated ball in the universal cover of the examples $(M,g)$ constructed in Theorem \ref{main} is shown. where the lifts of $(M,g)$'s drawstring are drawn in blue. }
\label{fig:pulledstringball}
\end{figure}

\vspace{1em}

\textbf{Related work and questions:} 
Beyond the classical results already mentioned, some recent findings suggested a positive answer for Conjecture \ref{AST}.
In an unpublished note of I. Agol \cite{agol}, it was shown that if the two lowest eigenvalues of $\mathrm{Ric}_g$ sum to at least $-4$, and if there is no cut locus in the universal cover, then the conclusion of the conjecture is true. Note that this curvature condition lies between $\Ric_g\geq-2g$ and $R_g\geq -6$.

By H. Davaux \cite{Davaux03}, if $M$ is a closed hyperbolic $3$-manifold and if $(M,g)$ satisfies $R_g\geq -6$, then \TODO{KX: checked the reference. OK!}
\begin{equation}\label{MW}
\lambda_1(\tilde{M},\tilde{g})\leq 1,
\end{equation}
where $\lambda_1(\tilde{M},\tilde g)$ is the first nonzero eigenvalue of the Laplacian on the universal cover $(\tilde{M},\tilde g)$ of $(M,g)$. More recently, that result was generalized to many complete $3$-manifolds in the work of O. Munteanu and J. Wang \cite{MW23,MW22}. 
On the other hand, it is well-known (by considering compactly supported $C^2$ approximations to the functions $e^{-s\dist_p(.)}$ for $s> h(g)/2$) that
\begin{equation}\label{eq:lambdaent}
    \lambda_1(\tilde{M},\tilde g)\leq \frac{1}{4}h(g)^2.
\end{equation}
Since Conjecture \ref{AST} and \eqref{eq:lambdaent} imply \eqref{MW}, one can consider \eqref{MW} as a positive solution to a weaker version of Conjecture \ref{AST}.

Finally, in \cite{CMN22}, D. Calegari, F.C. Marques and A. Neves introduced a notion of minimal surface entropy $E(g)$, based on a count of minimal surfaces instead of geodesics as for the volume entropy. B. Lowe and A. Neves \cite{LN21} proved that $R_g\geq-6$ implies
\begin{equation}
    E(g)\leq E(g_0)=2.
\end{equation}
One may consider this statement as the minimal surface version of Conjecture \ref{AST}. The proof of \cite{LN21} relies on the Ricci flow.


We do not know  whether Conjecture \ref{AST} holds 
{for metrics $C^{\infty}$-close to hyperbolic metrics,}
or for metrics satisfying some additional curvature condition. For instance, it is unknown \footnote{This question was suggested by Ben Lowe.} if $h(g)\leq2$ for negatively curved metrics satisfying $R_g\geq-6$. One may also ask if there is a threshold $A>1$ so that $R_g\geq-6$ and $\mathrm{Ric}_g\geq -2Ag$ implies $h(g)\leq 2$. These are in the same spirit as H. Bray's Football Theorem \cite[Theorem 18]{Braythesis}, where an additional Ricci curvature bound leads to a sharp volume comparison theorem for scalar curvature. Also, these curvature conditions prohibit the formation of drawstrings which are essential to the examples in Theorem \ref{main}.


Both \cite{KX} and \cite{S} were motivated by questions related to stability. First, in \cite{KX}, the first and third named authors found a drawstring construction which was used to provide counterexamples to conjectures \cite{Sormani21} on the geometric stability of Geroch conjecture. The drawstring example in \cite{KX} is roughly described as follows. Given a simple closed geodesic $\sigma$ in a $3$-torus, one can modify the flat metric in an arbitrarily small neighborhood of $\sigma$, so that (1) the length of $\sigma$ computed by the new metric is arbitrarily small, making the manifold Gromov-Hausdorff close to a pulled string metric space, (2) the scalar curvature is decreased only by a small amount. This construction is partly inspired by (but different from) the higher dimensional examples of collapsing tori due to M.C. Lee, A. Naber and R. Neumayer; see \cite{KX} for details. In Section \ref{drawst}, we will see that the drawstring construction can be performed in the hyperbolic setting.
Next, in \cite{S}, the second named author showed that the volume entropy inequality of Besson-Courtois-Gallot is stable for the measured Gromov-Hausdorff topology, after possibly removing  subsets with small boundary area.
A proof ingredient, which we will use later, is a sharp volume entropy comparison for Riemannian metrics almost metrically dominated by the hyperbolic metric. This comparison result was shown using the equidistribution properties of geodesic spheres in hyperbolic manifolds.

As an additional consequence of our construction, we will see that given a closed hyperbolic 3-manifold $(M,g_0)$, there are metrics $g$ with $R_g\geq -6$ and volumes arbitrarily close to $g_0$, but which are uniformly far from $g_0$ with respect to the Gromov-Hausdorff topology, the Gromov-Prokhorov topology or the intrinsic flat topology \cite{SW11}.
In particular, Perelman's volume bound \eqref{perelman} is not stable for those standard topologies. It is an open problem to find a more subtle stability statement for \eqref{perelman}. One possibility is to first find negligible subsets $Z\subset M$, then estimate the distance between $(M\setminus Z,g)$ and $(M,g_0)$ in the sense of metric measure spaces \cite{S,DS23} or perhaps using the $d_p$ topology defined by Lee-Naber-Neumayer \cite{LNN23}.

\subsection*{Acknowledgments}
\TODO{removed professor?}
The authors would like to thank Ian Agol, Conghan Dong, Christina Sormani for enlightening discussions, Xiaodong Wang for pointing out the reference \cite{Davaux03}, and Andr\'{e} Neves and Ben Lowe for their interests in this topic. 

A.S. was partially supported by NSF grant DMS-2104254. This research was conducted during the period A.S. served as a Clay Research Fellow.



\section{Drawstrings in hyperbolic 3-manifolds}
\label{drawst}

\begin{theorem}\label{thm:creating_ds}
    Let $(M,g_0)$ be a closed hyperbolic 3-manifold, and $\gamma$ be a simple closed geodesic. For any $a<0$, $\epsilon>0$, $r_0>0$, there exists $r_1<r_0$ and a metric $g$ such that:

    (i) $g=g_0$ outside $N_{g_0}(\gamma,r_1)$, and $\min\big\{(1-\epsilon)^2,e^{2a}\big\}g_0\leq g\leq e^{-2a}g_0$,

    (ii) $g|_\gamma=e^{2a}g_0|_\gamma$,

    (iii) $R_g\geq-6-\epsilon$,

    (iv) if $\sigma:[0,r_1]\to M$ is a unit speed $g_0$-geodesic with $\sigma(0)\in\gamma$ and $\sigma'(0)\perp\gamma$, then $\length_g(\sigma)\leq 3r_1$.
\end{theorem}

\noindent{\bf{Proof overview:}} The construction is a modification of the one carried out in \cite[Section 4]{KX}. It is organized as follows. We work in a cylindrical coordinate system near $\gamma$ and describe the metric coefficients of $g$ in terms of two radial functions. These functions are determined by four positive constants $r_1,r_2,c_1,c_2$, chosen in terms of $a,\varepsilon,r_0$, and the geometry of $\gamma$. After a preliminary estimate on $R_g$ and fixing appropriate values for these constants, properties (i)--(iv) are established.

\begin{proof}
    Fix two cutoff functions $\zeta,\eta\in C^\infty([0,\infty))$ satisfying
    \begin{align}\label{eq-cutoff:etazeta}
        \begin{split}
            & \zeta|_{[0,\frac12]}=0,\quad \zeta|_{[1,\infty)}=1,\quad    0\leq\zeta'\leq4,\quad |\zeta''|\leq16,\\
            & \eta|_{[0,\frac12]}=1,\quad \eta|_{[1,\infty)}=0,\quad 0\geq\eta'\geq-4, \quad|\eta''|\leq16.
        \end{split}
    \end{align}
    
    Choose $r_1<\min\big\{r_0,e^{-100}\big\}$ smaller than the normal injectivity radius of $\gamma$. In cylindrical coordinates near $\gamma$, the hyperbolic metric is
    \begin{equation}\label{eq:hyp_metric}
        g_0=dr^2+\sinh^2 rd\th^2+\cosh^2 rdt^2.
    \end{equation}
    For positive constants $r_2<r_1/64,c_1,c_2$ to be determined, we set the functions
    \begin{align}
        \psi(r)&=\int_0^r\Big[\zeta\big(\frac s{r_2}\big)\frac{1}{s\log^2(1/s)}+\big(1-\zeta\big(\frac s{r_2}\big)\big)\frac s{r_2}\Big]\,ds, \label{eq-cutoff:def_psi}\\
        h(r)&=1-c_1\eta\big(\frac r{r_1}\big)\psi(r), \label{eq-cutoff:def_h}\\
        u(r) &= -c_2\int_r^\infty\zeta\big(\frac s{4r_2}\big)\eta\big(\frac{4s}{r_1}\big)\frac{ds}{s\log(1/s)}.\label{eq-cutoff:def_u}
    \end{align}
    \vspace{-18pt}
    \begin{figure}[h!]
        \setlength{\abovecaptionskip}{-14pt}
        \setlength{\belowcaptionskip}{-6pt}
        \hspace{-153pt}\includegraphics[scale=1.15]{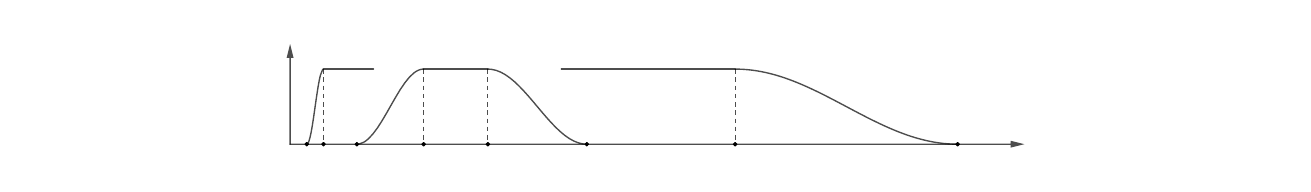}
        \begin{picture}(0,0)
            \put(28,90){$\zeta\big(\frac r{r_2}\big)$}
            \put(70,90){$\zeta\big(\frac{r}{4r_2}\big)\eta\big(\frac{4r}{r_1}\big)$}
            \put(190,90){$\eta\big(\frac r{r_1}\big)$}
            \put(2,28){$0$}
            \put(12,28){$\frac{r_2}2$}
            \put(24,28){$r_2$}
            \put(38,28){$2r_2$}
            \put(75,28){$4r_2$}
            \put(112,28){$\frac{r_1}8$}
            \put(166,28){$\frac{r_1}4$}
            \put(249,28){$\frac{r_1}2$}
            \put(374,28){$r_1$}
            \put(410,28){$r$}
        \end{picture}
        \caption{The cutoff functions that appear when defining $h,u$}
        \label{fig:enter-label}
        \TODO{before publishing: the graph should be on the same page as the cutoff functions.}
    \end{figure}
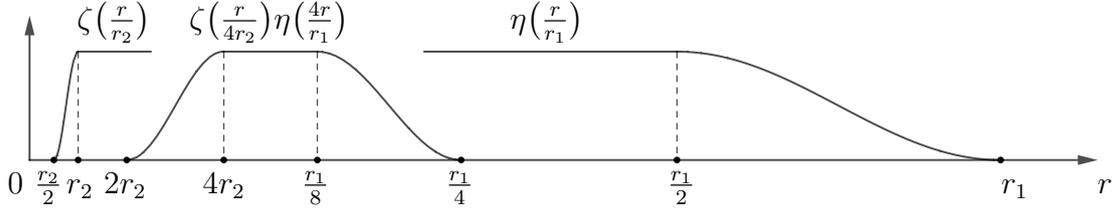

    \noindent With these functions, consider the drawstring metric
    \begin{equation}\label{eq:ds_metric}
        g=e^{-2u}dr^2+e^{-2u}h^2\sinh^2r d\th^2+e^{2u}\cosh^2rdt^2.
    \end{equation}
    Note that $g=g_0$ for $r>r_1$. Since $u$ is constant and $h=1-\frac{c_1}{2r_2}r^2$ near $r=0$, $g$ is a constant multiple of a doubly warped product
    \[dr^2+\big(1-\frac{c_1}{2r_2}r^2\big)\sinh^2rd\theta^2+e^{4u(0)}\cosh^2rdt^2\]
    in a neighborhood of $\gamma$. The coefficients of $d\theta^2$ and $dt^2$ satisfy the requisite conditions in \cite[Proposition 1.4.7]{Petersen}, hence $g$ is smooth near $\gamma$. A routine computation (postponed to the end of this section) shows
    \begin{equation}\label{eq:scal_formula_1}
        R_g=-6e^{2u}+2e^{2u}\Big[-\frac{h_{rr}}h-2\frac{h_r}{h}\coth r-u_r^2\Big]-2e^{2u}\tanh r\Big[\frac{h_r}h+2u_r\Big],
    \end{equation}
    where the subscripts $_r$ and $_{rr}$ denote the first and second derivatives with respect to $r$. Since $\big|\!\coth r-r^{-1}|\leq r$ and $\tanh r\leq r$ for $r\leq r_1$, we have
    \begin{equation}\label{eq:scal_formula}
        R_g\geq-6e^{2u}+2e^{2u}\Big[-\frac{h_{rr}}h-2\frac{h_r}{rh}-u_r^2\Big]-6e^{2u}r\Big[\frac{|h_r|}h+|u_r|\Big].
    \end{equation}

    We first let $c_1<\min\{r_1,\epsilon\}$. Note that for all $0\leq r\leq r_1$,
    \[\psi(r)\leq\int_0^r\frac{ds}{s\log^2(1/s)}+\int_0^{\min\{r,r_2\}}\frac s{r_2}\,ds\leq\frac1{\log(1/r_1)}+\frac{r_2}2<\frac12,\]
    so by \eqref{eq-cutoff:def_h}, we have the pointwise estimate $\max\{1-\epsilon,1/2\}\leq h\leq 1$ for all $r\in[0,r_1]$, which we will make frequent use of. Next we claim that by decreasing $c_1$, we may achieve $R_g\geq-6-\epsilon$ for all $r\in[r_1/4,r_1]$. To prove this, we note that $u(r)\equiv0$ in $[r_1/4,r_1]$, so \eqref{eq:scal_formula} is reduced to
    \[\begin{aligned}
        R_g &\geq -6+2\Big[-\frac{h_{rr}}h-\frac{2h_r}{rh}\Big]-6r\cdot\frac{|h_r|}h\geq -6-4|h_{rr}|-20r^{-1}|h_r|,
    \end{aligned}\]
    where we have used $h\geq1/2$. Moreover, we have $\zeta(r/r_2)\equiv1$ for $r\in[r_1/4,r_1]$, hence on this domain the quantities $h_r$ and $h_{rr}$ only involve $c_1$ and $r_1$ (but not $c_2,r_2$). Since $R_g\equiv-6$ when $c_1=0$, our claim follows by a continuity argument. We fix such a choice of $c_1$. The lower bound of $R_g$ on the remaining interval $[0,r_1/4]$ occurs below.

    Next, we claim that there exist choices $r_2<r_1/64$ and $c_2<c_1$, such that $u(0)=a$. Clearly, to achieve $u(0)=a$, one needs to set
    \[c_2=-a\Big[\int_0^\infty\eta\big(\frac{4s}{r_1}\big)\zeta\big(\frac s{4r_2}\big)\frac{ds}{s\log(1/s)}\Big]^{-1}.\]
    Our claim follows by observing that the right hand side converges to $0$ when $r_2\to0$ (since $\int_x^{r_1/8} ds/[s\log(1/s)]= \log\log x-\log\log(r_1/8)$). Let us fix these choices of $r_2,c_2$.

    Note that $u(0)=a$ implies property (ii) of the theorem. Also, we have $a\leq u\leq 0$ for all $r\leq r_1$. Thus, the second statement of property (i) follows by combining \eqref{eq:ds_metric} and $1-\epsilon\leq h\leq 1$. 

    Now we are ready to control $R_g$ in the range $r\leq r_1/4$. In this case note that $\eta(r/r_1)\equiv1$. Denote $w=\log(1/r)$. We may directly compute (similar to \cite[(126)--(130)]{KX})
    \begin{equation}\label{eq:scal_est_1}
        \begin{aligned}
            -\frac{h_{rr}}{h}-\frac{2h_r}{rh}-u_r^2 &= \frac{c_1}{1-c_1\psi}\Big[\zeta'\big(\frac r{r_2}\big)\big(\frac1{r_2rw^2}-\frac r{r_2^2}\big)+\frac{\zeta(r/r_2)}{r^2w^2} \\
            &\qquad +2\frac{\zeta(r/r_2)}{r^2w^3}+3\frac{1-\zeta(r/r_2)}{r_2}\Big]-\frac{c_2^2}{r^2w^2}\zeta^2\big(\frac r{4r_2}\big)\eta^2\big(\frac{4r}{r_1}\big),
        \end{aligned}
    \end{equation}
    which is nonnegative in $[0,r_1]$; the inequalities $\zeta'\geq0$ and $\frac1{rw^2}\geq1\geq\frac r{r_2}$ show the first term in the bracket is nonnegative, and the facts $c_2^2\leq c_1$, $\eta\leq1$, and $\zeta(r/4r_2)\leq\zeta(r/r_2)\leq1$ imply that the second term in the bracket dominates the final quantity.
    Also, we have $h\geq\frac12$, $|h_r|\leq c_1\big(\frac1{rw^2}+(1-\zeta(r/r_2))\frac r{r_2}\big)$ and $|u_r|<\frac{c_2}{rw}$, so
    \begin{equation}\label{eq:scal_est_2}
        6e^{2u}r\Big[\frac{|h_r|}h+|u_r|\Big]\leq 6r\cdot\Big[2c_1\frac1{rw^2}+2c_1+\frac{2c_2}{rw}\Big]\leq\frac{36\epsilon}{w}\leq\epsilon.
    \end{equation}
    Combined with \eqref{eq:scal_formula} and \eqref{eq:scal_est_1}, we obtain $R_g\geq-6-\epsilon$ for all $r\leq r_1/4$. Taking into account our choice of $c_1$, this establishes property (iii). 
    
    Finally, to prove property (iv), it is sufficient to show that
    \[\int_0^{r_1}e^{-u}\,dr\leq 3r_1.\]
    To see that this holds, note that $u(r)\geq -c_2\log\log(1/r)\geq-r_1\log\log(1/r)$, thus
    \begin{align}
        \int_0^{r_1}e^{-u}\,dr &\leq \int_0^{r_1}\big(\log(1/r)\big)^{r_1}\,dr \nonumber\\
        &= r_1\big(\log(1/r_1)\big)^{r_1}+r_1\int_0^{r_1}\big(\log(1/r)\big)^{r_1-1}\,dr. \label{eq:lengthint}
    \end{align}
    \TODO{KX: Christina pointed out to me that there is a calculation error here. And she suggests us to use $(\log(1/x))^{r_1}$ instead of $\log(1/x)^{r_1}$.}
    Using $(\log(1/x))^x\leq 2$ and $\log(1/r)>1$ for all $r\leq r_1$, we find that the right hand side of \eqref{eq:lengthint} is bounded by $3r_1$.
\end{proof}

To conclude, we provide the scalar curvature computations asserted in \eqref{eq:scal_formula_1}.

\begin{proof}[Proof of \eqref{eq:scal_formula_1}]
    Consider the surfaces of constant $r$-coordinate $\Sigma_r=\big\{d_{g_0}(\cdot,\gamma)=r\big\}$. Let $g_r,A,H$ be the induced metric, second fundamental form, and mean curvature (i.e the sum of the principal curvatures) with respect to $g$ of $\Sigma_r$, and denote the unit normal $\nu=e^u\p_r$. Since $\Sigma_r$ is flat, the Gauss equation and the variational formula for mean curvature yield
    \[0=R_g-2\Ric_g(\nu,\nu)+H^2-|A|^2_{g_r},\qquad e^uH_r=-|A|^2_{g_r}-\Ric_g(\nu,\nu).\]
    Cancelling the Ricci curvature, we obtain
    \begin{equation}\label{eq:R1}
        R_g=-\big(H^2+|A|^2+2e^uH_r\big).
    \end{equation}
    From the general formula $A=\frac12(\mathcal{L}_\nu g_r)|_{\Sigma_r}$, we may calculate
    \[\begin{aligned}
        H &= e^u\big(h^{-1}h_r+\tanh r+\coth r\big), \\
        |A|^2_{g_r} &= e^{2u}\big(\!-u_r+h^{-1}h_r+\coth r\big)^2+e^{2u}\big(u_r+\tanh r\big)^2.
    \end{aligned}\]
    Inserting into \eqref{eq:R1} and using the hyperbolic trigonometry identities, this implies
    \[\begin{aligned}
       R_g &= -e^{2u}\Big[\frac{h_r^2}{h^2}+(\tanh r+\coth r)^2+2\frac{h_r}h(\tanh r+\coth r) \\
       &\qquad\qquad +u_r^2+\frac{h_r^2}{h^2}+\coth^2r-2u_r\frac{h_r}h-2u_r\coth r+2\frac{h_r}h\coth r \\
       &\qquad\qquad +u_r^2+\tanh^2r+2u_r\tanh r\\
       &\qquad\qquad +2u_r\Big(\frac{h_r}h+\tanh r+\coth r\Big)+2\frac{h_{rr}}h-2\frac{h_r^2}{h^2}+\frac2{\cosh^2 r}-\frac2{\sinh^2 r}\Big] \\
       &= -e^{2u}\Big[2\frac{h_{rr}}h+4\frac{h_r}h\coth r+2u_r^2+2\frac{h_r}h\tanh r+4u_r\tanh r+6\Big]. \qedhere
   \end{aligned}\]
\end{proof}


\section{Volume entropy comparison}
\label{comparer}
We are ready to make use of the drawstring construction from the previous section to prove the main theorem, Theorem \ref{main}. We will construct two types of counterexamples to the Agol-Storm-Thurston problem. The first type of metrics have very large volume entropy, and the proof is based on a general systolic lower bound for the volume entropy \cite[Corollary 3]{BM23}. The second type of metrics are $C^0$-close to the hyperbolic metric, and the proof relies on an asymptotic volume entropy comparison \cite[Theorem 3.5]{S}. 

\subsection{Examples with arbitrarily large volume entropy}

\begin{theorem}[identical to Theorem \ref{main}(A)]\label{main1}
    Let $(M,g_0)$ be a closed hyperbolic 3-manifold. Then $M$ admits a Riemannian metric $g$ with $R_g\geq -6$, and with arbitrarily large volume entropy $h(g)$.
\end{theorem}

A first tool is the following ``collar lemma'', which yields a bound for the volume entropy, see \cite[Theorem 1.11]{BCGS17} \cite[Theorem 1.2]{Cerocchi14} \cite[Corollary 2]{BM23}. We will use the statement of Balacheff-Merlin \cite{BM23}, which is a direct consequence of Lim's work \cite{Lim08}:
\begin{theorem}\label{collar}
    Let $(N,g_N)$ be a closed Riemannian manifold with positive volume entropy $h(g_N)>0$. Suppose that $c_1$, $c_2$ are two loops based at a point $x$ in $N$, which generate a free group of rank $2$ inside the fundamental group of $N$. 
    Denote $l_i:=\length(c_i)$. Then we have
    \[\frac1{1+e^{h(g_N)l_1}}+\frac1{1+e^{h(g_N)l_2}}\leq\frac12.\]
    
\end{theorem}

In particular, this collar lemma tells us that if $l_1$ is small, then either $l_2$ or $h(g_N)$ is large. Next, we will also need the following standard property of fundamental groups of hyperbolic manifolds, which is a consequence of the ping-pong lemma:


\begin{lemma} \label{ping pong}
Let $(M,g_0)$ be a closed hyperbolic manifold and let $\Gamma$ be its fundamental group. For any two noncommuting elements $q_1,q_2\in \Gamma$, there exists an integer $A>0$ such that $q_1^A$ and $q_2^A$ generate a free group of rank $2$ inside $\Gamma$.
\end{lemma}

\begin{proof}[Proof of Theorem \ref{main1}]
    Let $\gamma$ be a simple closed geodesic in $(M,g_0)$, for instance a $g_0$-length minimizing representative of a class in $\pi_1(M)$. Let $g_i$ be a sequence of metrics on $M$ given by Theorem \ref{thm:creating_ds} with parameters $a\to -\infty$, $\epsilon=1$, and $r_0\ll1$ fixed. In particular, by Theorem \ref{thm:creating_ds} (ii),
    $$\lim_{i\to \infty}\length_{g_i}(\gamma)=0.$$
    Fixing a basepoint $x\in \gamma\subset  M$, we can identify $\gamma$ with an element $q_1$ of the fundamental group $\Gamma$ of $M$. By Lemma \ref{ping pong} (and since $\Gamma$ is not abelian), there is another element $q_2\in \Gamma$ and an integer $A$ such that $q_1^A$, $q_2^A$ generate a free group of rank $2$ inside $\Gamma$. Going $A$ times around the curve $\gamma$, we easily find a loop $c_{1}$ 
    based at $x$ representing $q_1^A$, such that
    \begin{equation}\label{eq:lengthc1}
        \lim_{i\to \infty}\length_{g_i}(c_{1})=0.
    \end{equation}
    On the other hand, we can find a smoothly embedded loop $c_2$ representing $q_2^A$, such that $\{x\}=c_2\cap\gamma$, $c_2\perp\gamma$ at $x$, and $c_2\cap N_{g_0}(\gamma,r_0)$ is a connected $g_0$-geodesic segment of $g_0$-length $2r_0$ (namely, the only connected component of $c_2\cap N_{g_0}(\gamma,r_0)$ passes through $x$). For $i$ sufficiently large, by Theorem \ref{thm:creating_ds} (i) and (iv), we have
    \begin{equation}\label{eq:lengthc2}
    \length_{g_i}(c_2)\leq 3\length_{g_0}(c_2) <\infty.
    \end{equation}

    We conclude by combining \eqref{eq:lengthc1} and \eqref{eq:lengthc2} with Theorem \ref{collar}: since the $g_i$-length of $c_{1}$ tends to $0$, and the $g_i$-length of $c_{2}$ stays uniformly bounded, it follows that $h(g_i)\to\infty$. According to Theorem \ref{thm:creating_ds}, we have $R_{g_i}\geq-7$, and thus a suitable scaling of the sequence $g_i$ satisfy the conditions of Theorem \ref{main1}.
\end{proof}

\subsection{Examples arbitrarily close to being hyperbolic}
\begin{theorem}[identical to Theorem \ref{main}(B)]\label{main2}
Let $(M,g_0)$ be a closed hyperbolic 3-manifold. Then $M$ admits a Riemannian metric $g$ arbitrarily close to $g_0$ in the $C^0$ topology, such that $R_g\geq -6$ and {$h(g)>2$}.
\end{theorem}

This time the proof makes use of the following volume entropy comparison theorem proved by the second named author \cite[Theorem 3.5]{S}:

\begin{theorem} \label{cle}
Let $(M,g_0)$ be a closed oriented hyperbolic manifold of dimension $n\geq 2$. Suppose that the following hold:
\begin{enumerate}
\item there are Riemannian metrics $g_i$ ($i\geq 1$) on $M$, and a (not necessarily Riemannian) metric $d$ on $M$, so that the length metric $d_{g_i}$ of $g_i$ converges in $C^0$ to $d$ as $i\to\infty$,
\item the identity map $\id: (M,d_{g_0}) \to (M,d)$ is $1$-Lipschitz and bi-Lipschitz.        
\end{enumerate}
\noindent If $\id: (M,d_{g_0}) \to (M,d)$ is not an isometry, then
\begin{equation}
    \liminf_{i\to \infty} h(g_i) >h(g_0)=n-1.
\end{equation}
\end{theorem}

\begin{proof}[Proof of Theorem \ref{main2}]
    Fix $\delta>0$ and let $(M,g_0)$ be as in the theorem. Fix a simple closed geodesic $\gamma$ in $(M,g_0)$, and let $g_i$ be the sequence of metrics on $M$ given by Theorem \ref{thm:creating_ds} with parameters $a=-\delta$, $\epsilon=r_0=1/i$, so that for all $i>0$ large enough,
    \begin{equation} \label{eq:2sided}
        R_{g_i}\geq-6-i^{-1},\qquad e^{-2\delta}g_0\leq g_i\leq e^{2\delta}g_0.
    \end{equation}
    Let $d_{g_i}$ denote the length metric of $g_i$. We claim the following:
    \begin{enumerate}[label={(\Roman*)}, nosep]
        \item some subsequence of $d_{g_i}$ converges in $C^0$ to a metric $d$ on $M$,
        \item the identity map $\id:(M,d_{g_0})\to(M,d)$ is not an isometry,
        \item the identity map $\id$ is bi-Lipschitz and 1-Lipschitz.
    \end{enumerate}
    
    

    \noindent With these facts, we can apply Theorem \ref{cle} to $g_i$ and deduce  
    \begin{equation}
        \liminf_{i\to \infty} h(g_i) >h(g_0)=2.
    \end{equation}
    Taking  $\delta$ sufficiently small, this readily implies (after a slight rescaling) that there are Riemannian metrics on $M$ arbitrarily $C^0$-close to $g_0$, with scalar curvature at least $-6$, but with volume entropy strictly larger than $2$. 
    
    It remains to prove the three claims above.

    \vspace{1em}
    
    \noindent\textit{Proof of (I).} 
    By \eqref{eq:2sided}, we have
    $$e^{-2\delta} d_{g_0} \leq d_{g_i} \leq e^{2\delta} d_{g_0}.$$
    Thus $\{d_{g_i}\}$ is an equicontinuous family of continuous functions on $M\times M$ with respect to $d_{g_0}$.
    So (I) follows readily from the Arzela-Ascoli theorem. Moreover, the limit metric $d$ satisfies
       \begin{equation}\label{eq:bill}
      e^{-2\delta} d_{g_0} \leq d \leq e^{2\delta} d_{g_0}.
    \end{equation}

        \vspace{1em}
        
    \noindent\textit{Proof of (II).} This follows from $g_i|_\gamma=e^{-2\delta}g|_\gamma$ and the $C^0$-convergence.

    \vspace{1em}
    
    \noindent\textit{Proof of (III).} The bi-Lipschitz property of the identity map follows from \eqref{eq:bill}. Proving $1$-Lipschitzness requires a bit more checking. Let $s_I$ be the normal injectivity radius of $\gamma$ for the metric $g_0$. We can assume that $i^{-1}<{s_I/2}$. For convenience, we define a $g_0$-orthogonal projection map $\pi$: for $p\in\overline{N(\gamma,i^{-1})}$, we let $\pi(p)\in\gamma$ be the point with the shortest $g_0$-distance to $p$. We note the following facts about $g_0$ and $\pi$:

 \vspace{1em}

    \noindent\textbf{Fact 1.} (a) If $\sigma:[0,1]\to M\setminus N(\gamma,i^{-1})$ is a $g_0$-geodesic segment such that both $\sigma(0)$ and $\sigma(1)$ lie on $\p N(\gamma,i^{-1})$, then $\length_{g_0}(\sigma)\geq s_I/2$.

    (b) If $p\in\p N(\gamma,i^{-1})$ then $d_{g_i}(p,\pi(p))<3i^{-1}$.

    (c) The map $\pi:N(\gamma,i^{-1})\to\gamma$ is 1-Lipschitz with respect to both $g_0$ and $g_i$.


    
    \vspace{3pt}
    \noindent\textit{Proof of Fact 1.} (a) Since $\p N(\gamma,s)$ is strictly convex {for all $s\leq s_I$}, the function $x\mapsto d_{g_0}(x,\gamma)$ is strictly convex in $N(\gamma,s_I)$. Thus $\sigma$ must leave $N(\gamma,s_I)$, otherwise $d(-,\gamma)$ has a local maximum on $\sigma$, which is not possible. The statement immediately follows for large $i$. Item (b) follows from Theorem \ref{thm:creating_ds} (iv), and item (c) {follows directly from the metric expressions \eqref{eq:hyp_metric} \eqref{eq:ds_metric} and the fact that the warping factor defined by \eqref{eq-cutoff:def_u} satisfies  $u(0)\leq u(r)$ for $r\geq0$.} 
    \\


The desired $1$-Lipschitzness will be an immediate corollary of the following:

\vspace{6pt}
    \noindent\textbf{Fact 2.} There is a constant $C$ such that $d_{g_i}\leq d_{g_0}+Ci^{-1}$.
    
    \vspace{3pt}
    \noindent\textit{Proof of Fact 2.} Let $x,y\in M$ and $\sigma$ be a shortest $g_0$-geodesic from $x$ to $y$. There exists times $0=b_0\leq a_1<b_1<a_2<\cdots<a_n<b_n\leq a_{n+1}=l$, such that $\sigma|_{(a_k,b_k)}\subset N({\gamma},i^{-1})$ and $\sigma|_{[b_k,a_{k+1}]}\subset M\setminus N(\gamma,i^{-1})$. Recall that $g_i=g_0$ on $M\setminus N(\gamma,i^{-1})$. In the case $\sigma\subset M\setminus N(\gamma,i^{-1})$ we have by Theorem \ref{thm:creating_ds} (i): $$d_{g_i}(x,y)\leq\length_{g_i}(\sigma)=\length_{g_0}(\sigma)=d_{g_0}(x,y).$$ 
    
    Now assume that $\sigma$ enters $N(\gamma,i^{-1})$ at least once. For each interval $(a_k,b_k)$, we let $p_k=\pi(\sigma(a_k))$ and $q_k=\pi(\sigma(b_k))$. Using Theorem \ref{thm:creating_ds}(ii) and Fact 1(c), we can estimate
    \begin{align}
        \length_{g_0}(\sigma([a_k,b_k])) &\geq \length_{g_0}\pi(\sigma([a_k,b_k])))  \\
        &\geq d_{g_0|_{{\gamma}}}(p_k,q_k)
        \geq d_{g_i|_{{\gamma}}}(p_k,q_k)
        \geq d_{g_i}(p_k,q_k) \\
        &\geq d_{g_i}(\sigma(a_k),\sigma(b_k))-d_{g_i}(\sigma(a_k),p_k)-d_{g_i}(q_k,\sigma(b_k)) \\
        &\geq d_{g_i}(\sigma(a_k),\sigma(b_k))-6i^{-1},
    \end{align}
    where we used Fact 1(b) in the last inequality. We also have by Theorem \ref{thm:creating_ds} (i):
    \begin{align}
        \length_{g_0}(\sigma([b_k,a_{k+1}]))=\length_{g_i}(\sigma([b_k,a_{k+1}]))\geq d_{g_i}(\sigma(b_k),\sigma(a_{k+1})).
    \end{align}
    Summing the above inequalities, we have 
    \begin{equation}
        d_{g_0}(x,y)\geq d_{g_i}(x,y)-6ni^{-1}.
    \end{equation}
    Since $\length_{g_0}(\sigma)\leq\diam(M,g_0)$ and $\length_{g_0}(\sigma([b_k,a_{k+1}]))\geq s_I/2$ for each $1\leq k<n$ by Fact 1(a), we obtain $n\leq 2\diam(M)/s_I$. Hence 
    \begin{align}
        d_{g_i}(x,y)\leq d_{g_i}(x,y)+\frac{12\operatorname{diam}(M)}{s_I}\cdot\frac{1}{i}
    \end{align}
    and Fact 2 follows.
\end{proof}

\bibliographystyle{alpha}
\bibliography{JEMS_final/pulledstringentropy}

\noindent\textit{Department of Mathematics, Michigan State University,
East Lansing, MI 48824}

\noindent\textit{Email: \href{mailto:kazarasd@msu.edu}{kazarasd@msu.edu}}

\vspace{9pt}

\noindent\textit{Department of Mathematics, Duke University, Durham, NC, 27708,}

\noindent\textit{Email: \href{mailto:kx35@math.duke.edu}{kx35@math.duke.edu}}

\vspace{9pt}

\noindent\textit{Department of Mathematics, California Institute of Technology\\ 177 Linde Hall, \#1200 E. California Blvd., Pasadena, CA 91125}

\noindent\textit{Email: \href{mailto:aysong@caltech.edu}{aysong@caltech.edu}}

\end{document}